\documentclass[12pt,bezier]{article}

\usepackage{mathrsfs}
\usepackage{graphics}
\usepackage{latexsym}
\usepackage{amsmath}
\usepackage{amsfonts}
\usepackage{amssymb}
\usepackage{epsf}
\usepackage{cases}
\usepackage{mathtools}
\usepackage{tabularx}
\usepackage{latexsym,bm}
\usepackage{amsthm}
\usepackage{cite}
\usepackage{graphicx,epsfig,subfigure}

\newtheorem{lem}{Lemma}[section]
\newtheorem{thm}[lem]{Theorem}

\begin{document}

\title{On the existence of specified cycles in bipartite tournaments}
\author{ Bo Zhang,\quad  Weihua Yang\footnote{Corresponding author. E-mail: ywh222@163.com, yangweihua@tyut.edu.cn.}\\
\small Department of Mathematics, Taiyuan University of
Technology, Taiyuan 030024,
China\\
}
\date{}
\maketitle

{\small{\bf Abstract.}\quad For two integers $n\geq 3$ and $2\leq p\leq n$, we denote $D(n,p)$ the digraph  obtained from a directed $n$-cycle by changing the orientations of $p-1$ consecutive arcs. In this paper, we show that a family of $k$-regular $(k\geq 3)$ bipartite tournament $BT_{4k}$ contains $D(4k,p)$ for all $2\leq p\leq 4k$ unless $BT_{4k}$ is isomorphic to a digraph $D$ such that $(1,2,3,...,4k,1)$ is a Hamilton cycle and $(4m+i-1,i)\in A(D)$ and $(i,4m+i+1)\in A(D)$, where $1\leq m\leq k-1$.


\section{Introduction}

For terminology not explicitly introduced here, we refer to \cite{bang,four}. We denote the in-degree and out-degree of vertex $x$ by $d^{-}(x)$ and $d^{+}(x)$ respectively in digraphs. Let $D$ be a digraph. If $xy$ is an arc of $D$, we say that $x$ dominates $y$. A digraph $D$ is $k$-regular if $d^{+}(v)=d^{-}(v)=k$ for all $v\in V$. Clearly, a $k$-regular bipartite tournament has $4k$ vertices.

Let $n$, $p$ be two integers, $n\geq 3$, $2\leq p\leq n$. We denote by $D(n,p)$ the digraph obtained from a directed $n$-cycle by changing $p-1$ consecutive arcs. We use $[a_{1}, a_{2},...,a_{p}; a_{1}, b_{1}, b_{2},...,b_{n-p}, a_{p}]$ to denote the $D(n,p)$ consisting of a directed path $(a_{1}, a_{2},...,a_{p})$ of length $p-1$ and a directed path $(a_{1}, b_{1}, b_{2},...,b_{n-p}, a_{p})$ of length $n-p+1$, where $a_{i}\neq b_{j}$, for all $i$, $j$. It is easy to see that $D(n,p)$ and $D(n,n-p+2)$ are isomorphic.

We denote by $P(n,p)=[a_{1}, a_{2},...,a_{p}; b_{1}, b_{2},...,b_{n-p}, a_{p}]$ the directed graph obtained from $D(n,p)$ by removing the arc $(a_1,b_1)$, and we call vertices $a_1,b_1$ the initial vertex and the end vertex of $P(n,p)$, respectively. Analogously, we define the digraph $Q(n,p)=[a_{1}, a_{2},...,a_{p};$ $a_{1}, b_{1}, b_{2},...,b_{n-p}]$ (i.e. delete the arc $(b_{n-1},a_p)$) from $D(n,p)$), with initial vertex $a_{p}$ and end vertex $b_{n-p}$. Alspach and Rosenfeld \cite{one} proved that every tournament of order $n\geq 4$ contains $P(n,p)$ and $Q(n,p)$ for $p=2,3,...,n-1$. The result was applied to prove the existence of $D(n,p)$ in tournaments by Wojda.

Note that a tournament has a directed Hamiltonian path or a directed Hamiltonian cycle. Instead of just looking for directed Hamiltonian paths or directed Hamiltonian cycles, one can seek arbitrary orientations of paths and cycles. An oriented path is a sequence of vertices $a_{1}, a_{2},..., a_{n}$ in a tournament $T$, the type of this path is characterized by the sequence $\sigma_{n-1}=\varepsilon_{1}...\varepsilon_{n-1}$ where $\varepsilon_{i}=+1$ if $a_{i}\rightarrow a_{i+1}$ and $\varepsilon_{i}=-1$ if $a_{i+1}\rightarrow a_{i}$. A block in a type $\sigma_{n-1}=\varepsilon_{1}...\varepsilon_{n-1}$ is a maximal subsequence of consecutive elements having the same sign (maximal with respect to inclusion). An oriented cycle is defined similarly.

A simple path or cycle in a digraph is said to be anti-directed if every two adjacent arcs have opposing orientations. An anti-directed Hamiltonian path of a digraph, which contains all the vertices, has the largest number of blocks. An digraph of odd order cannot have an anti-directed Hamiltonian cycle. The problems about anti-directed Hamiltonian Paths or cycles were discussed early. In 1971, Grunbaum \cite{two} initiated the study of anti-directed paths and cycles in tournaments, he proved that all tournaments of order $n\geq 8$ have an anti-directed Hamiltonian path. In 1972, Rosenfeld \cite{liu} proved that all tournaments of order $n\geq 12$ have an anti-directed Hamiltonian path starting from any specified vertex. Later, Hell and Rosenfeld \cite{P.H} proved that in any tournament there is an anti-directed Hamiltonian path from a specified first vertex to a specified last vertex, and starting with an arc of specified direction, as long as both the first vertex and the last vertex satisfy an obvious minimally necessary condition, they apply this result to give a polynomial time algorithm to find an anti-directed Hamiltonian path between two specified vertices of an input tournament. Grunbaum conjectured that all even tournaments of order $2n\geq 10$ have an anti-directed Hamiltonian cycle. In 1973, Thomasen \cite{C.Th} proved it for $2n\geq 50$ and the next year Rosenfeld \cite{eleven} improved the bound to $2n\geq 28$. In 1983, Petrovic \cite{Petrovic} proved that all even tournaments of order $2n\geq 16$ have an anti-directed Hamiltonian cycle.

Since the total degree of every vertex in a tournament on $2n$ vertices is $2n-1$, Grant wondered if all digraphs on $2n$ vertices with total degree $2n-1$ have an anti-directed Hamiltonian cycle. In 1980, Grant \cite{DDG} proved that if $D$ is a directed graph on $2n$ vertices and $\delta^{0}(D)\geq \frac{4}{3}n+2\sqrt{n\log n}$, then $D$ has an anti-directed Hamiltonian cycle. In 1995, Haggkvist and Thomason \cite{three} proved that if $D$ is a digraph on $n$ vertices and $\delta(D)\geq \frac{n}{2}+n^{\frac{5}{6}}$, then $D$ contains every orientation of a cycle on $n$ vertices. In 2008, Plantholt and Tipnis \cite{six} improved the result by showing that if $D$ is a digraph on $2n$ vertices and $\delta(D)\geq \frac{4}{3}n$, then $D$ has an anti-directed Hamiltonian cycle. In 2013, Busch et al. \cite{seven} improved the results by showing that if $D$ is a digraph on $2n$ vertices and $\delta(D)\geq n+\frac{14}{3}\sqrt{n}$, then $D$ has an anti-directed Hamiltonian cycle. Later, DeBiasio and Molla \cite{LD} proved that if $D$ is a digraph on $n$ vertices with minimum out-degree and in-degree at least $\frac{n}{2}+1$, then $D$ has an anti-directed Hamiltonian cycle.

In 1974, Rosenfeld conjectured that there is an integer $N>8$ such that every tournament of order $n\geq N$ contains every nondirected cycle of order $n$. Rosenfeld's cycle conjecture was proved in full generally by Thomason \cite{ten}, who proved that a tournament $T$ of order $n\geq 2^{128}$ contains a non-strongly oriented cycle of order $n$. While Thomason made no attempt to sharpen this bound, he indicated that it should be true for tournaments of order at least 9. The cycles with just two blocks is another special case. In 1971, Grunbaum \cite{two} and Thomassen \cite{eight,nine} studied the existence of $D(n,2)$ in a tournament. Grunbaum proved that every tournament $T$ of order $n\geq 3$ contains $D(n,2)$ unless $T$ is isomorphic to $T_{3}^{*}$ or $T_{5}^{*}$. Thomassen proved that the minimum number of copies of $D(n,2)$ in a strong tournament is at least $n-5$, and at least $c^{n}$ for some constant $c>1$. In 1983, Benhocine and Wojda \cite{Benhoc} proved that every digraph with $n\geq 3$ and at least $(n-1)(n-2)+2$ arcs contains $D(n,p)$ for $p=2,3,...,n$ with two exceptions. Wojda \cite{five} proved that every tournament of order $n\geq 7$ contains $D(n,p)$ for $p=2,3,...,n$ and the tournaments of order $n$, $3\leq n\leq 6$, which do not contain $D(n,p)$ for some $p$.

Motivated by considering cycles with just two blocks in tournaments, we explore to generalize the problem about $D(n,p)$ to bipartite tournaments. In this paper, the decomposable $k$-regular bipartite tournament will be defined by induction, as follows:

1. A decomposable $1$-regular bipartite tournament $BT_{4}$ is a cycle of length 4.

2. A decomposable $2$-regular bipartite tournament $BT_{8}$ is obtained from two decomposable $1$-regular bipartite tournaments by adding  arcs.

3. A decomposable $k$-regular bipartite tournament $BT_{4k}$ is obtained from a decomposable $1$-regular bipartite tournament $BT_{4}$ and a decomposable $(k-1)$-regular bipartite tournament $BT_{4(k-1)}$ by adding arcs.

We prove that a decomposable $2$-regular bipartite tournament either contains only $D(8,5)$ or contains $D(8,p)$ for $2\leq p\leq 8$ and $p\neq 5$, and that a decomposable $3$-regular bipartite tournament $BT_{12}$ contains $D(12,p)$ for $2\leq p\leq 12$ unless $BT_{12}$ is isomorphic to a digraph $D$ which has a Hamiltonian cycle $(1,2,3,...,12,1)$ and  $(4m+i-1,i)\in A(D)$ and $(i,4m+i+1)\in A(D)$  for all $i$. Furthermore, we prove that a decomposable $k$-regular $(k\geq 3)$ bipartite tournament $BT_{4k}$ contains $D(4k,p)$ for $2\leq p\leq 4k$ unless $BT_{4k}$ is isomorphic to a digraph $D$ having a Hamiltonian cycle $(1,2,3,...,4k,1)$ and  $(4m+i-1,i)\in A(D)$ and $(i,4m+i+1)\in A(D)$ for all $i$.

\section{preliminary results}

In this paper, we let $V=V(D)=\{1,2,3,...,n\}$ be the set of vertices of $D$ and denote a cycle by $(1,2,3,...,n,1)$. Throughout the paper, the number should take modulo $n$ if it necessary.

\begin{lem}\label{1} Let $BT_{8}$ be a decomposable $2$-regular bipartite tournament, then $BT_{8}$ either contains only $D(8,5)$ or contains $D(8,p)$ for $2\leq p\leq 8$ and $p\neq 5$.
\end{lem}

\begin{proof} Note that $BT_{8}$ is consisted of two decomposable $1$-regular bipartite tournaments, see figure 1.

Case 1. When $4\rightarrow5$ and $8\rightarrow1$, then $5\rightarrow2\rightarrow7\rightarrow4$ and $1\rightarrow6\rightarrow3\rightarrow8$ (see figure 2). We have $D(8,5)=[5,2,3,4,1; 5,6,7,8,1]$. Now, we shall show that it has no $D(8,p)$, $2\leq p\leq4$ and $6\leq p\leq8$, by contradiction.
Suppose $BT_{8}$ has $D(8,2)=D(8,8)=[v_{1},v_{8}; v_{1},v_{2},v_{3},v_{4},v_{5},v_{6},v_{7},v_{8}]$. Since $BT_{8}$ is 2-regular, we have $v_{8}\rightarrow v_{5}\rightarrow v_{2}\rightarrow v_{7}\rightarrow v_{4}\rightarrow v_{1}$ and $v_{8}\rightarrow v_{3}\rightarrow v_{6}\rightarrow v_{1}$. So we just get two Hamiltonian cycles $(v_{1},v_{8},v_{5},v_{2},v_{3},v_{6},v_{7},v_{4},v_{1})$ and $(v_{1},v_{2},v_{7},v_{8},v_{3},v_{4},v_{5},v_{6},v_{1})$, but there is $v_{1}\rightarrow v_{2}$ and $v_{1}\rightarrow v_{8}$ respectively, they are not isomorphic to the digraph of figure 2. Then $BT_{8}$ has no $D(8,2)$ and $D(8,8)$. By the similar way, we prove that $BT_{8}$ also has no $D(8,3)$, $D(8,4)$, $D(8,6)$, and $D(8,7)$.

Case 2. When $4\rightarrow5$ and $1\rightarrow8$, then $5\rightarrow2\rightarrow7\rightarrow4$ and $8\rightarrow3\rightarrow6\rightarrow1$ (see figure 3). We have $D(8,2)=D(8,8)=[1,8; 1,2,3,4,5,6,7,8]$, $D(8,3)=D(8,7)=[1,8,5; 1,2,3,6,7,4,5]$ and $D(8,4)=D(8,6)=[1,2,3,4; 1,8,5,6,7,4]$. By the similar way as in case 1, we prove that $BT_{8}$ has no $D(8,5)$.

When $5\rightarrow4$ and $1\rightarrow8$, the graph is isomorphic to the digraph of figure 2 and it is similar to case 1. Also, when $5\rightarrow4$ and $8\rightarrow1$, the graph is isomorphic to the digraph of figure 3 and it is similar to case 2.

Thus, the result is proved.
\end{proof}

\begin{figure}[ht]
\label{fig-exgraph2}
\center
\includegraphics[scale=1]{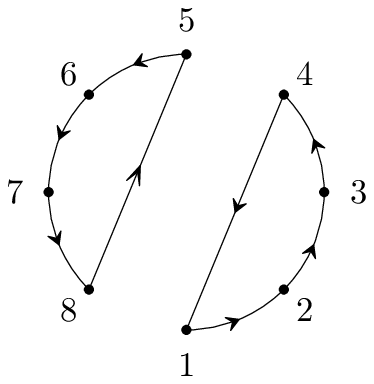}
\center
\small{figure 1}
\end{figure}

\begin{lem}\label{2} Let $BT_{12}$ be a decomposable $3$-regular bipartite tournament, then $BT_{12}$ contains $D(12,p)$ for $2\leq p\leq 12$ unless $BT_{12}$ is isomorphic to a digraph $D$ containing a Hamiltonian cycle $(1,2,3,...,12,1)$ and  $(4m+i-1,i)\in A(D)$ and $(i,4m+i+1)\in A(D)$ for all $i$.
\end{lem}

\begin{proof} $BT_{12}$ is consisted of a decomposable $1$-regular bipartite tournament $BT_{4}$ and a decomposable $2$-regular bipartite tournament $BT_{8}$. $BT_{8}$ has two cases as figure 2 and figure 3. Let $BT_{4}$ be a cycle: $1^{'}\rightarrow2^{'}\rightarrow3^{'}\rightarrow4^{'}\rightarrow1^{'}$.

\begin{figure}[ht]
\label{fig-exgraph2}
\center
\includegraphics[scale=1]{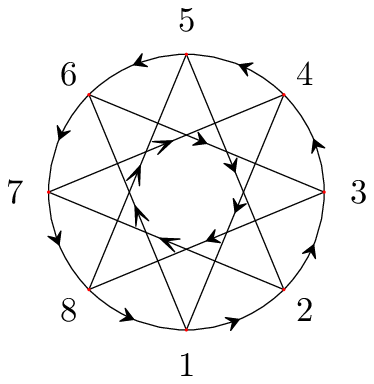}
\center
\small{figure 2}
\end{figure}

\begin{figure}[ht]
\label{fig-exgraph2}
\center
\includegraphics[scale=1]{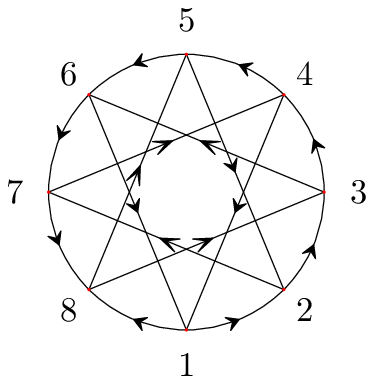}
\center
\small{figure 3}
\end{figure}

Case 1. $BT_{8}$ is isomorphic to the digraph in figure 3.

We first prove that $BT_{12}$ contains $D(12,p)$ for $2\leq p\leq 4$, $10\leq p\leq 12$, $p=6$ and $p=8$. By Lemma \ref{1}, $BT_{8}$ has $D(8,2)=D(8,8)=[1,8; 1,2,3,4,5,6,7,8]$. The vertex 2 must dominates one of the vertices in $BT_{4}$, without loss of generality, we assume that $2\rightarrow1^{'}$, then $3^{'}\rightarrow2$.

When $4^{'}\rightarrow3$, we have $1\rightarrow2\rightarrow1^{'}\rightarrow2^{'}\rightarrow3^{'}\rightarrow4^{'}\rightarrow3\rightarrow4\rightarrow5\rightarrow6\rightarrow7\rightarrow8$. When $3\rightarrow4^{'}$. If $3^{'}\rightarrow4$, we have $1\rightarrow2\rightarrow3\rightarrow4^{'}\rightarrow1^{'}\rightarrow2^{'}\rightarrow3^{'}\rightarrow4\rightarrow5\rightarrow6\rightarrow7\rightarrow8$; if $4\rightarrow3^{'}$, then $1^{'}\rightarrow4$. Now we consider two cases as follows.

$(a)$ When $1\rightarrow4^{'}$, then $2^{'}\rightarrow1$. As $3^{'}\rightarrow2$, we have $1\rightarrow4^{'}\rightarrow1^{'}\rightarrow2^{'}\rightarrow3^{'}\rightarrow2\rightarrow3\rightarrow4\rightarrow5\rightarrow6\rightarrow7\rightarrow8$.

$(b)$ When $4^{'}\rightarrow1$, then $1\rightarrow2^{'}$. As $3^{'}\rightarrow2$, $3\rightarrow4^{'}$ and $1^{'}\rightarrow4$, we have $1\rightarrow2^{'}\rightarrow3^{'}\rightarrow2\rightarrow3\rightarrow4^{'}\rightarrow1^{'}\rightarrow4\rightarrow5\rightarrow6\rightarrow7\rightarrow8$.

Then $BT_{12}$ contains $D(12,2)$ and $D(12,12)$. By Lemma \ref{1}, $BT_{8}$ has $D(8,3)=D(8,7)$ and $D(8,4)=D(8,6)$. In the same way, we have $BT_{12}$ contains $D(12,3)$, $D(12,4)$, $D(12,6)$, $D(12,8)$, $D(12,10)$ and $D(12,11)$.

Now, we shall prove that $BT_{12}$ contains $D(12,5)$, $D(12,7)$ and $D(12,9)$ in this case.

Case 1.1 When $1^{'}\rightarrow4$, we have $4\rightarrow3^{'}$. If $7\rightarrow 2^{'}$, then $4^{'}\rightarrow7$. Now we have $4\rightarrow3^{'}\rightarrow4^{'}\rightarrow1^{'}\rightarrow2^{'}$ and $4\rightarrow1\rightarrow8\rightarrow5\rightarrow2\rightarrow3\rightarrow6\rightarrow7\rightarrow2^{'}$; $7\rightarrow8\rightarrow5\rightarrow2\rightarrow3\rightarrow6\rightarrow1$ and $7\rightarrow2^{'}\rightarrow3^{'}\rightarrow4^{'}\rightarrow1^{'}\rightarrow4\rightarrow1$; that's to say we have $D(12,5)=D(12,9)=[4,3^{'},4^{'},1^{'},2^{'}; 4,1,8,5,2,3,6,7,2^{'}]$ and $D(12,7)=[7,8,5,2,3,6,1; 7,2^{'},3^{'},4^{'},1^{'},4,1]$ in this case. If $2^{'}\rightarrow7$, then $7\rightarrow4^{'}$. Now we consider two cases.

Case 1.1.1 When $1\rightarrow2^{'}$, then $4^{'}\rightarrow1$. We have $4\rightarrow3^{'}\rightarrow4^{'}\rightarrow1^{'}\rightarrow2^{'}$ and $4\rightarrow5\rightarrow2\rightarrow7\rightarrow8\rightarrow3\rightarrow6\rightarrow1\rightarrow2^{'}$; $4\rightarrow3^{'}\rightarrow4^{'}\rightarrow1^{'}\rightarrow2^{'}\rightarrow7\rightarrow8$ and $4\rightarrow5\rightarrow2\rightarrow3\rightarrow6\rightarrow1\rightarrow8$; that's to say we have $D(12,5)=D(12,9)=[4,3^{'},4^{'},1^{'},2^{'}; 4,5,2,7,8,3,6,1,2^{'}]$ and $D(12,7)=[4,3^{'},4^{'},1^{'},2^{'},7,8; 4,5,2,3,6,1,8]$ in this case.

Case 1.1.2 When $2^{'}\rightarrow1$, then $1\rightarrow4^{'}$. We have $2^{'}\rightarrow3^{'}\rightarrow4^{'}\rightarrow1^{'}\rightarrow4$ and $2^{'}\rightarrow1\rightarrow8\rightarrow5\rightarrow2\rightarrow3\rightarrow6\rightarrow7\rightarrow4$; $4\rightarrow3^{'}\rightarrow4^{'}\rightarrow1^{'}\rightarrow2^{'}\rightarrow1\rightarrow8$ and $4\rightarrow5\rightarrow2\rightarrow3\rightarrow6\rightarrow7\rightarrow8$; that's to say we have $D(12,5)=D(12,9)=[2^{'},3^{'},4^{'},1^{'},4; 2^{'},1,8,5,2,3,6,7,4]$ and $D(12,7)=[4,3^{'},4^{'},1^{'},2^{'},1,8; 4,5,2,3,6,7,8]$ in this case.

Case 1.2 When $4\rightarrow1^{'}$, then $3^{'}\rightarrow4$. If $7\rightarrow 2^{'}$, then $4^{'}\rightarrow7$. Now we have $7\rightarrow2^{'}\rightarrow3^{'}\rightarrow4^{'}\rightarrow1^{'}$ and $7\rightarrow8\rightarrow5\rightarrow6\rightarrow1\rightarrow2\rightarrow3\rightarrow4\rightarrow1^{'}$; $4\rightarrow5\rightarrow2\rightarrow3\rightarrow6\rightarrow1\rightarrow8$ and $4\rightarrow1^{'}\rightarrow2^{'}\rightarrow3^{'}\rightarrow4^{'}\rightarrow7\rightarrow8$; that's to say we have $D(12,5)=D(12,9)=[7,2^{'},3^{'},4^{'},1^{'}; 7,8,5,6,1,2,3,4,1^{'}]$ and $D(12,7)=[4,5,2,3,6,1,8; 4,1^{'},2^{'},3^{'},4^{'},7,8]$ in this case. If $2^{'}\rightarrow7$, then $7\rightarrow4^{'}$. We have $4\rightarrow1^{'}\rightarrow2^{'}\rightarrow3^{'}\rightarrow4^{'}$ and $4\rightarrow1\rightarrow8\rightarrow5\rightarrow2\rightarrow3\rightarrow6\rightarrow7\rightarrow4^{'}$, that's to say we have $D(12,5)=D(12,9)=[4,1^{'},2^{'},3^{'},4^{'}; 4,1,8,5,2,3,6,7,4^{'}]$ in this case. Now we consider two cases for $p=7$ in this case.

Case 1.2.1 When $4^{'}\rightarrow1$, then $1\rightarrow2^{'}$. We have $7\rightarrow4\rightarrow1^{'}\rightarrow2^{'}\rightarrow3^{'}\rightarrow4^{'}\rightarrow1$ and $7\rightarrow8\rightarrow5\rightarrow2\rightarrow3\rightarrow6\rightarrow1$, that's to say we have $D(12,7)=[7,4,1^{'},2^{'},3^{'},4^{'},1; 7,8,5,2,3,6,1]$ in this case.

Case 1.2.2 When $1\rightarrow4^{'}$, then $2^{'}\rightarrow1$. We have $3^{'}\rightarrow4^{'}\rightarrow1^{'}\rightarrow2^{'}\rightarrow1\rightarrow2\rightarrow3$ and $3^{'}\rightarrow4\rightarrow5\rightarrow6\rightarrow7\rightarrow8\rightarrow3$, that's to say we have $D(12,7)=[3^{'},4^{'},1^{'},2^{'},1,2,3; 3^{'},4,5,6,7,8,3]$ in this case.

Case 2. $BT_{8}$ is isomorphic to the digraph in figure 2.

Since there is a vertex in $BT_{4}$ dominates a vertex in $BT_{8}$, without loss of generality, we assume that $4^{'}\rightarrow1$. Then we have $1\rightarrow2^{'}$.

Case 2.1 When $8\rightarrow1^{'}$, then $3^{'}\rightarrow8$.

Case 2.1.1 When $3^{'}\rightarrow2$, then $2\rightarrow1^{'}$ and it also has $1^{'}\rightarrow4\rightarrow3^{'}$, $1^{'}\rightarrow6\rightarrow3^{'}$. We have
$4^{'}\rightarrow1$ and $4^{'}\rightarrow1^{'}\rightarrow2^{'}\rightarrow3^{'}\rightarrow2\rightarrow3\rightarrow4\rightarrow5\rightarrow6\rightarrow7\rightarrow8\rightarrow1$; $3^{'}\rightarrow4^{'}\rightarrow1^{'}\rightarrow2^{'}$ and $3^{'}\rightarrow2\rightarrow3\rightarrow4\rightarrow5\rightarrow6\rightarrow7\rightarrow8\rightarrow1\rightarrow2^{'}$; $4^{'}\rightarrow1^{'}\rightarrow2^{'}\rightarrow3^{'}\rightarrow8$ and $4^{'}\rightarrow1\rightarrow2\rightarrow3\rightarrow4\rightarrow5\rightarrow6\rightarrow7\rightarrow8$; that's to say, we have $D(12,2)=D(12,12)=[4^{'},1; 4^{'},1^{'},2^{'},3^{'},2,3,4,5,6,7,8,1]$, $D(12,4)=D(12,10)=[3^{'},4^{'},1^{'},2^{'}; 3^{'},2,3,4,5,6,7,8,1,2^{'}]$ and $D(12,5)=D(12,9)=[4^{'},1^{'},2^{'},3^{'},8; 4^{'},1,2,3,4,5,6,7,8]$ in this case. Now we consider two cases as follows in this case.

Case 2.1.1$(a)$ When $4^{'}\rightarrow3$, then $3\rightarrow2^{'}$, $5\rightarrow4^{'}$, $7\rightarrow4^{'}$, $2^{'}\rightarrow5$ and $2^{'}\rightarrow7$. We have $2^{'}\rightarrow7\rightarrow4$ and $2^{'}\rightarrow5\rightarrow6\rightarrow3^{'}\rightarrow4^{'}\rightarrow1\rightarrow2\rightarrow3\rightarrow8\rightarrow1^{'}\rightarrow4$; $2^{'}\rightarrow7\rightarrow8\rightarrow1\rightarrow2\rightarrow3$ and $2^{'}\rightarrow3^{'}\rightarrow4^{'}\rightarrow1^{'}\rightarrow4\rightarrow5\rightarrow6\rightarrow3$; $2^{'}\rightarrow5\rightarrow6\rightarrow7\rightarrow8\rightarrow1^{'}\rightarrow4$ and $2^{'}\rightarrow3^{'}\rightarrow4^{'}\rightarrow1\rightarrow2\rightarrow3\rightarrow4$; that's to say, we have $D(12,3)=D(12,11)=[2^{'},7,4; 2^{'},5,6,3^{'},4^{'},1,2,3,8,1^{'},4]$, $D(12,6)=D(12,8)=[2^{'},7,8,1,2,3; 2^{'},3^{'},4^{'},1^{'},4,5,6,3]$ and $D(12,7)=[2^{'},3^{'},4^{'},1,2,3,4; 2^{'},5,6,7,8,1^{'},4]$ in this case.

Case 2.1.1$(b)$ When $3\rightarrow4^{'}$, then $2^{'}\rightarrow3$. We have $2^{'}\rightarrow3^{'}\rightarrow4^{'}\rightarrow1\rightarrow2\rightarrow1^{'}$ and
$2^{'}\rightarrow3\rightarrow4\rightarrow5\rightarrow6\rightarrow37\rightarrow8\rightarrow1^{'}$; $1^{'}\rightarrow4\rightarrow5\rightarrow6\rightarrow7\rightarrow8\rightarrow1$ and $1^{'}\rightarrow2^{'}\rightarrow3^{'}\rightarrow2\rightarrow3\rightarrow4^{'}\rightarrow1$; that's to say, we have $D(12,6)=D(12,8)=[2^{'},3,4,5,6,7,8,1^{'}; 2^{'},3^{'},4^{'},1,2,1^{'}]$ and $D(12,7)=[1^{'},2^{'},3^{'},2,3,4^{'},1; 1^{'},4,5,6,7,8,1]$ in this case. Now if $5\rightarrow4^{'}$, then $4^{'}\rightarrow7\rightarrow2^{'}\rightarrow5$, we have $5\rightarrow4^{'}\rightarrow7$ and $5\rightarrow6\rightarrow3^{'}\rightarrow8\rightarrow1^{'}\rightarrow2^{'}\rightarrow3\rightarrow4\rightarrow1\rightarrow2\rightarrow7$; if $4^{'}\rightarrow5$, then $5\rightarrow2^{'}\rightarrow7\rightarrow4^{'}$, we have $2^{'}\rightarrow3\rightarrow4$ and $2^{'}\rightarrow3^{'}\rightarrow4^{'}\rightarrow5\rightarrow6\rightarrow7\rightarrow8\rightarrow1\rightarrow2\rightarrow1^{'}\rightarrow4$. So we have $D(12,3)=D(12,11)=[5,4^{'},7; 5,6,3^{'},8,1^{'},2^{'},3,4,1,2,7]$ or $D(12,3)=D(12,11)=[2^{'},3,4; 2^{'},3^{'},4^{'},5,6,7,8,1,2,1^{'},4]$ in this case.

Case 2.1.2 When $2\rightarrow3^{'}$, then $1^{'}\rightarrow2$. We have $4^{'}\rightarrow1^{'}\rightarrow2^{'}\rightarrow3^{'}\rightarrow8$ and $4^{'}\rightarrow1\rightarrow2\rightarrow3\rightarrow4\rightarrow5\rightarrow6\rightarrow7\rightarrow8$, that's to say, we have $D(12,5)=D(12,9)=[4^{'},1^{'},2^{'},3^{'},8; 4^{'},1,2,3,4,5,6,7,8]$ in this case.

Case 2.1.2$(a)$ When $7\rightarrow2^{'}$, then $4^{'}\rightarrow7$, $2^{'}\rightarrow3$, $2^{'}\rightarrow5$, $3\rightarrow4^{'}$ and $5\rightarrow4^{'}$. We have $8\rightarrow1^{'}$ and $8\rightarrow1\rightarrow2\rightarrow3\rightarrow4\rightarrow5\rightarrow6\rightarrow7\rightharpoonup2^{'}\rightarrow3^{'}\rightarrow4^{'}\rightarrow1$; $3\rightarrow4\rightarrow1\rightarrow2$ and $3\rightarrow8\rightarrow5\rightarrow6\rightarrow7\rightarrow2^{'}\rightarrow3^{'}\rightarrow4^{'}\rightarrow1^{'}\rightarrow2$; $3\rightarrow4\rightarrow5\rightarrow6\rightarrow7\rightarrow2^{'}$ and $3\rightarrow8\rightarrow1\rightarrow2\rightarrow3^{'}\rightarrow4^{'}\rightarrow1^{'}\rightarrow2^{'}$; that's to say, we have $D(12,2)=D(12,12)=[8,1^{'}; 8,1,2,3,4,5,6,7,2^{'},3^{'},4^{'},1]$, $D(12,4)=D(12,10)=[3,4,1,2; 3,8,5,6,7,2^{'},3^{'},4^{'},1^{'},2]$ and $D(12,6)=D(12,8)=[3,4,5,6,7,2^{'}; 3,8,1,2,3^{'},4^{'},1^{'},2^{'}]$ in this case. If $4\rightarrow3^{'}$, then $1^{'}\rightarrow4$, $6\rightarrow1^{'}$ and $3^{'}\rightarrow6$. Now we have $8\rightarrow1^{'}\rightarrow4$ and $8\rightarrow5\rightarrow6\rightarrow7\rightarrow2^{'}\rightarrow3^{'}\rightarrow4^{'}\rightarrow1\rightarrow2\rightarrow3\rightarrow4$; $2^{'}\rightarrow5\rightarrow6\rightarrow7\rightarrow8\rightarrow1^{'}\rightarrow4$ and $2^{'}\rightarrow3^{'}\rightarrow4^{'}\rightarrow1\rightarrow2\rightarrow3\rightarrow4$; that's to say, now we have $D(12,3)=D(12,11)=[8,1^{'},4; 8,5,6,7,2^{'},3^{'},4^{'},1,2,3,4]$ and $D(12,7)=[2^{'},5,6,7,8,1^{'},4; 2^{'},3^{'},4^{'},1,2,3,4]$ in this case. If $3^{'}\rightarrow4$, then $4\rightarrow1^{'}\rightarrow6\rightarrow3^{'}$. We have $4\rightarrow1^{'}\rightarrow2^{'}$ and $4\rightarrow1\rightarrow2\rightarrow3\rightarrow8\rightarrow5\rightarrow6\rightarrow3^{'}\rightarrow4^{'}\rightarrow7\rightarrow2^{'}$; $3^{'}\rightarrow4\rightarrow1^{'}\rightarrow2^{'}\rightarrow5\rightarrow6\rightarrow7$ and $3^{'}\rightarrow8\rightarrow1\rightarrow2\rightarrow3\rightarrow4^{'}\rightarrow7$; that's to say, now we have $D(12,3)=D(12,11)=[4,1^{'},2^{'}; 4,1,2,3,8,5,6,3^{'},4^{'},7,2^{'}]$ and $D(12,7)=[3^{'},4,1^{'},2^{'},5,6,7; 3^{'},8,1,2,3,4^{'},7]$ in this case.

Case 2.1.2$(b)$ When $2^{'}\rightarrow7$, then $7\rightarrow4^{'}$.

$(i)$ When $4^{'}\rightarrow3$, then $3^{'}\rightarrow2^{'}\rightarrow5\rightarrow4^{'}$. We have $8\rightarrow1^{'}$ and $8\rightarrow1\rightarrow6\rightarrow7\rightarrow4\rightarrow5\rightarrow2\rightarrow3\rightarrow2^{'}\rightarrow3^{'}\rightarrow4^{'}\rightarrow1^{'}$; $1^{'}\rightarrow2^{'}\rightarrow3^{'}\rightarrow4^{'}$ and $1^{'}\rightarrow2\rightarrow3\rightarrow4\rightarrow1\rightarrow6\rightarrow7\rightarrow8\rightarrow5\rightarrow4^{'}$; $8\rightarrow1^{'}\rightarrow2^{'}\rightarrow3^{'}\rightarrow4^{'}\rightarrow3$ and $8\rightarrow5\rightarrow6\rightarrow7\rightarrow4\rightarrow1\rightarrow2\rightarrow3$; that's to say, we have $D(12,2)=D(12,12)=[8,1^{'}; 8,1,6,7,4,5,2,3,2^{'},3^{'},4^{'},1^{'}]$, $D(12,4)=D(12,10)=[1^{'},2^{'},3^{'},4^{'}; 1^{'},2,3,4,1,6,7,8,5,4^{'}]$ and $D(12,6)=D(12,8)=[8,1^{'},2^{'},3^{'},4^{'},3; 8,5,6,7,4,1,2,3]$ in this case. If $3^{'}\rightarrow4$, then $4\rightarrow1^{'}\rightarrow6\rightarrow3^{'}$. We have $4\rightarrow1^{'}\rightarrow2^{'}$ and $4\rightarrow5\rightarrow6\rightarrow7\rightarrow8\rightarrow1\rightarrow2\rightarrow3^{'}\rightarrow4^{'}\rightarrow3\rightarrow2^{'}$; $3^{'}\rightarrow4\rightarrow1^{'}\rightarrow2^{'}\rightarrow5\rightarrow6\rightarrow7$ and $3^{'}\rightarrow4^{'}\rightarrow3\rightarrow8\rightarrow1\rightarrow2\rightarrow7$; that's to say, now we have $D(12,3)=D(12,11)=[4,1^{'},2^{'}; 4,5,6,7,8,1,2,3^{'},4^{'},3,2^{'}]$ and $D(12,7)=[3^{'},4,1^{'},2^{'},5,6,7; 3^{'},4^{'},3,8,1,2,7]$ in this case. If $4\rightarrow3^{'}$, then $1^{'}\rightarrow4$, $6\rightarrow1^{'}$ and $3^{'}\rightarrow6$. Now we have $1^{'}\rightarrow4\rightarrow5$ and $1^{'}\rightarrow2\rightarrow3^{'}\rightarrow4^{'}\rightarrow1\rightarrow6\rightarrow3\rightarrow2^{'}\rightarrow7\rightarrow8\rightarrow5$; $2^{'}\rightarrow5\rightarrow6\rightarrow7\rightarrow8\rightarrow1^{'}\rightarrow4$ and $2^{'}\rightarrow3^{'}\rightarrow4^{'}\rightarrow1\rightarrow2\rightarrow3\rightarrow4$; that's to say, now we have $D(12,3)=D(12,11)=[1^{'},4,5; 1^{'},2,3^{'},4^{'},1,6,3,2^{'},7,8,5]$ and $D(12,7)=[2^{'},5,6,7,8,1^{'},4; 2^{'},3^{'},4^{'},1,2,3,4]$ in this case.

$(ii)$ When $3\rightarrow4^{'}$, then $4^{'}\rightarrow5\rightarrow2^{'}\rightarrow3$. If $4\rightarrow3^{'}$, then $1^{'}\rightarrow4$, $6\rightarrow1^{'}$ and $3^{'}\rightarrow6$. Now we have $1^{'}\rightarrow2^{'}$ and $1^{'}\rightarrow2\rightarrow7\rightarrow8\rightarrow5\rightarrow6\rightarrow3\rightarrow4\rightarrow3^{'}\rightarrow4^{'}\rightarrow1\rightarrow2^{'}$; $4^{'}\rightarrow1^{'}\rightarrow2$ and $4^{'}\rightarrow5\rightarrow2^{'}\rightarrow3\rightarrow4\rightarrow3^{'}\rightarrow6\rightarrow7\rightarrow8\rightarrow1\rightarrow2$; $1^{'}\rightarrow2\rightarrow3^{'}\rightarrow4^{'}$ and $1^{'}\rightarrow4\rightarrow5\rightarrow6\rightarrow7\rightarrow8\rightarrow1\rightarrow2^{'}\rightarrow3\rightarrow4^{'}$; $4^{'}\rightarrow1^{'}\rightarrow2^{'}\rightarrow3^{'}\rightarrow8$ and $4^{'}\rightarrow1\rightarrow2\rightarrow3\rightarrow4\rightarrow5\rightarrow6\rightarrow7\rightarrow8$; $4^{'}\rightarrow1^{'}\rightarrow2\rightarrow7\rightarrow8\rightarrow1$ and $4^{'}\rightarrow5\rightarrow2^{'}\rightarrow3^{'}\rightarrow6\rightarrow3\rightarrow4\rightarrow1$; $4^{'}\rightarrow5\rightarrow2^{'}\rightarrow3\rightarrow8\rightarrow1\rightarrow6$ and $4^{'}\rightarrow1^{'}\rightarrow2\rightarrow7\rightarrow4\rightarrow3^{'}\rightarrow6$; that's to say, now we have $D(12,2)=D(12,12)=[1^{'},2^{'}; 1^{'},2,7,8,5,6,3,4,3^{'},4^{'},1,2^{'}]$, $D(12,3)=D(12,11)=[4^{'},1^{'},2; 4^{'},5,2^{'},3,4,3^{'},6,7,8,1,2]$, $D(12,4)=D(12,10)=[1^{'},2,3^{'},4^{'}; 1^{'},4,5,6,7,8,1,2^{'},3,4^{'}]$, $D(12,5)=D(12,9)=[4^{'},1^{'},2^{'},3^{'},8; 4^{'},1,2,3,4,5,6,7,8]$ in this case, we also have $D(12,6)=D(12,8)=[4^{'},1^{'},2,7,8,1; 4^{'},5,2^{'},3^{'},6,3,4,1]$ and $D(12,7)=[4^{'},5,2^{'},3,8,1,6; 4^{'},1^{'},2,7,4,3^{'},6]$ in this case. If $3^{'}\rightarrow4$, $BT_{12}$ has a Hamiltonian cycle $(1,2,3,4,5,6,7,8,1^{'},2^{'},3^{'},4^{'})$ and $BT_{12}$ is isomorphic to a digraph $D$ which has a Hamiltonian cycle $(1,2,3,...,12,1)$ and for any vertex $i$, there are $(4m+i-1,i)\in A(D)$ and $(i,4m+i+1)\in A(D)$, where $1\leq m\leq 2$, the vertices modulo $12$ so that the vertex $12+i$ is the vertex $i$.

Case 2.2 When $1^{'}\rightarrow8$, then $8\rightarrow3^{'}$. We have $1^{'}\rightarrow8$ and $1^{'}\rightarrow2^{'}\rightarrow3^{'}\rightarrow4^{'}\rightarrow1\rightarrow2\rightarrow3\rightarrow4\rightarrow5\rightarrow6\rightarrow7\rightarrow8$; $4^{'}\rightarrow1^{'}\rightarrow2^{'}\rightarrow3^{'}$ and $4^{'}\rightarrow1\rightarrow2\rightarrow3\rightarrow4\rightarrow5\rightarrow6\rightarrow7\rightarrow8\rightarrow3^{'}$; $1\rightarrow2^{'}\rightarrow3^{'}\rightarrow4^{'}\rightarrow1^{'}\rightarrow8$ and $1\rightarrow2\rightarrow3\rightarrow4\rightarrow5\rightarrow6\rightarrow7\rightarrow8$; that's to say, we have $D(12,2)=D(12,12)=[1^{'},8; 1^{'},2^{'},3^{'},4^{'},1,2,3,4,5,6,7,8]$, $D(12,4)=D(12,10)=[4^{'},1^{'},2^{'},3^{'}; 4^{'},1,2,3,4,5,6,7,8,3^{'}]$ and $D(12,6)=D(12,8)=[1,2^{'},3^{'},4^{'},1^{'},8; 1,2,3,4,5,6,7,8]$ in this case. Now we consider two cases for others.

Case 2.2.1 When $3^{'}\rightarrow2$, then $2\rightarrow1^{'}$. We have $1^{'}\rightarrow6\rightarrow7\rightarrow8\rightarrow5$ and $1^{'}\rightarrow2^{'}\rightarrow3^{'}\rightarrow4^{'}\rightarrow1\rightarrow2\rightarrow3\rightarrow4\rightarrow5$; $4^{'}\rightarrow1^{'}\rightarrow8$ and $4^{'}\rightarrow1\rightarrow2^{'}\rightarrow3^{'}\rightarrow2\rightarrow3\rightarrow4\rightarrow5\rightarrow6\rightarrow7\rightarrow8$; that's to say, we have $D(12,5)=D(12,9)=[1^{'},6,7,8,5; 1^{'},2^{'},3^{'},4^{'},1,2,3,4,5]$ and $D(12,3)=D(12,11)=[4^{'},1^{'},8; 4^{'},1,2^{'},3^{'},2,3,4,5,6,7,8]$ in this case. Now we consider two cases as follows.

Case 2.2.1$(a)$ When $1^{'}\rightarrow6$, then $6\rightarrow3^{'}\rightarrow4\rightarrow1^{'}$. If $3\rightarrow4^{'}$, then $2^{'}\rightarrow3$. We have $2^{'}\rightarrow3^{'}\rightarrow4^{'}\rightarrow1\rightarrow2\rightarrow1^{'}\rightarrow8$ and $2^{'}\rightarrow3\rightarrow4\rightarrow5\rightarrow6\rightarrow7\rightarrow8$, that's to say, we have $D(12,7)=[2^{'},3^{'},4^{'},1,2,1^{'},8; 2^{'},3,4,5,6,7,8]$ in this case. If $4^{'}\rightarrow3$, then $3\rightarrow2^{'}$, $2^{'}\rightarrow5\rightarrow4^{'}$ and $2^{'}\rightarrow7\rightarrow4^{'}$.
We have $3\rightarrow2^{'}\rightarrow3^{'}\rightarrow4^{'}\rightarrow1\rightarrow2\rightarrow1^{'}$ and $3\rightarrow8\rightarrow5\rightarrow6\rightarrow7\rightarrow4\rightarrow1^{'}$; that's to say, we have $D(12,7)=[3,2^{'},3^{'},4^{'},1,2,1^{'}; 3,8,5,6,7,4,1^{'}]$ in this case.

Case 2.2.1$(b)$ When $6\rightarrow1^{'}$, then $1^{'}\rightarrow4\rightarrow3^{'}\rightarrow6$. If $3\rightarrow4^{'}$, then $2^{'}\rightarrow3$. We have $2^{'}\rightarrow3^{'}\rightarrow4^{'}\rightarrow1\rightarrow2\rightarrow1^{'}\rightarrow8$ and $2^{'}\rightarrow3\rightarrow4\rightarrow5\rightarrow6\rightarrow7\rightarrow8$, that's to say, we have $D(12,7)=[2^{'},3^{'},4^{'},1,2,1^{'},8; 2^{'},3,4,5,6,7,8]$ in this case. If $4^{'}\rightarrow3$, then $3\rightarrow2^{'}$, $2^{'}\rightarrow5\rightarrow4^{'}$ and $2^{'}\rightarrow7\rightarrow4^{'}$.
We have $1^{'}\rightarrow2^{'}\rightarrow7\rightarrow8\rightarrow5\rightarrow6\rightarrow3$ and $1^{'}\rightarrow4\rightarrow3^{'}\rightarrow4^{'}\rightarrow1\rightarrow2\rightarrow3$; that's to say, we have $D(12,7)=[1^{'},4,3^{'},4^{'},1,2,3; 1^{'},2^{'},7,8,5,6,3]$ in this case.

Case 2.2.2 When $2\rightarrow3^{'}$, then $1^{'}\rightarrow2$, $3^{'}\rightarrow4\rightarrow1^{'}$ and $3^{'}\rightarrow6\rightarrow1^{'}$. We have  $4^{'}\rightarrow1^{'}\rightarrow8$ and $4^{'}\rightarrow1\rightarrow2^{'}\rightarrow3^{'}\rightarrow6\rightarrow3\rightarrow4\rightarrow5\rightarrow2\rightarrow7\rightarrow8$; $6\rightarrow3\rightarrow4\rightarrow5\rightarrow2$ and $6\rightarrow7\rightarrow8\rightarrow1\rightarrow2^{'}\rightarrow3^{'}\rightarrow4^{'}\rightarrow1^{'}\rightarrow2$ in this case. Now we consider two cases as follows.

Case 2.2.2$(a)$ When $4^{'}\rightarrow7$, then $7\rightarrow2^{'}$. We have $1^{'}\rightarrow2\rightarrow3\rightarrow4\rightarrow5\rightarrow6\rightarrow7$ and $1^{'}\rightarrow8\rightarrow1\rightarrow2^{'}\rightarrow3^{'}\rightarrow4^{'}\rightarrow7$; that's to say, we have $D(12,7)=[1^{'},2,3,4,5,6,7; 1^{'},8,1,2^{'},3^{'},4^{'},7]$ in this case.

Case 2.2.2$(b)$ When $7\rightarrow4^{'}$, then $2^{'}\rightarrow7$. If $2^{'}\rightarrow5$, we have $7\rightarrow4^{'}\rightarrow1\rightarrow2\rightarrow3\rightarrow4\rightarrow5$ and $7\rightarrow8\rightarrow3^{'}\rightarrow6\rightarrow1^{'}\rightarrow2^{'}\rightarrow5$; that's to say, we have $D(12,7)=[7,4^{'},1,2,3,4,5; 7,8,3^{'},6,1^{'},2^{'},5]$ in this case. If $5\rightarrow2^{'}$, then $2^{'}\rightarrow3\rightarrow4^{'}\rightarrow5$. we have $4^{'}\rightarrow5\rightarrow2^{'}\rightarrow3^{'}\rightarrow6\rightarrow7\rightarrow8$ and $4^{'}\rightarrow1\rightarrow2\rightarrow3\rightarrow4\rightarrow1^{'}\rightarrow8$; that's to say, we have $D(12,7)=[4^{'},5,2^{'},3^{'},6,7,8; 4^{'},1,2,3,4,1^{'},8]$ in this case.
The proof is complete.
\end{proof}

\section{Main results}

\begin{thm} Let $BT_{4k}$ be a decomposable $k$-regular $(k\geq 3)$ bipartite tournament, then $BT_{4k}$ contains $D(4k,p)$ for all $2\leq p\leq 4k$ unless $BT_{4k}$ is isomorphic to a digraph $D$ having a Hamiltonian cycle $(1,2,3,...,4k,1)$ and  $(4m+i-1,i)\in A(D)$ and $(i,4m+i+1)\in A(D)$ for all $i$.
\end{thm}

\begin{proof} We know that $BT_{4k}$ is consisted of a decomposable $1$-regular bipartite tournament $BT_{4}$ and a decomposable $(k-1)$-regular bipartite tournament $BT_{4(k-1)}$. Let $BT_{4}$ be a cycle: $1^{'}\rightarrow2^{'}\rightarrow3^{'}\rightarrow4^{'}\rightarrow1^{'}$. We prove the theorem by induction on $k$. By Lemma \ref{2}, the result is true when $k=3$. Suppose that it holds for decomposable $(k-1)$-regular bipartite tournaments with $k\geq 4$. So $BT_{4(k-1)}$ contains $D(4k-4,p)$ for $2\leq p\leq 4k-4$ unless $BT_{4(k-1)}$ is isomorphic to a digraph $D$ which has a Hamiltonian cycle $(1,2,3,...,4k-4,1)$ and for any vertex $i$, there are $(4m+i-1,i)\in A(D)$ and $(i,4m+i+1)\in A(D)$, where $1\leq m\leq k-2$, the vertices modulo $4k-4$ so that the vertex $4k-4+i$ is the vertex $i$. We shall show that $BT_{4k}$ contains $D(4k,p)$ for $2\leq p\leq 4k$ unless $BT_{4k}$ is isomorphic to a digraph $D$ which has a Hamiltonian cycle $(1,2,3,...,4k,1)$ and for any vertex $i$, there are $(4m+i-1,i)\in A(D)$ and $(i,4m+i+1)\in A(D)$, where $1\leq m\leq k-1$, the vertices modulo $4k$ so that the vertex $4k+i$ is the vertex $i$. Furthermore, as $D(4k,p)$ and $D(4k,4k-p+2)$ are isomorphic, discussion on $D(4k,p)$ can be limited to $2\leq p\leq 2k+1$. Now we consider two cases as follows.

Case 1. $BT_{4(k-1)}$ contains $D(4k-4,p)$ for $2\leq p\leq 4k-4$.

We have $D(4k-4,p)=[a_{1},a_{2},...,a_{p}; a_{1},b_{1},...,b_{4k-4-p},a_{p}]$ $(2\leq p\leq4k-4)$. Let $P$ be a path such that $P=(a_{1},b_{1},...,b_{4k-4-p},a_{p})$. $P$ has at least five vertices when $2\leq p\leq 2k+1$. Selecting four consecutive adjacent points $a_{1}$, $b_{1}$, $b_{2}$, $b_{3}$. The vertex $a_{1}$ must dominates one of the vertices in $BT_{4}$, without loss of generality, we assume that $a_{1}\rightarrow1^{'}$, then $3^{'}\rightarrow a_{1}$ due to $BT_{4k}$ is $k$-regular. When $4^{'}\rightarrow b_{1}$, we have $a_{1}\rightarrow1^{'}\rightarrow2^{'}\rightarrow3^{'}\rightarrow4^{'}\rightarrow b_{1}\rightarrow b_{2}\rightarrow b_{3}\rightarrow...$, then $BT_{4}$ can be added into the path $P$ in this case. When $b_{1}\rightarrow4^{'}$, then $2^{'}\rightarrow b_{1}$. If $3^{'}\rightarrow b_{2}$, we have $a_{1}\rightarrow b_{1}\rightarrow4^{'}\rightarrow1^{'}\rightarrow2^{'}\rightarrow3^{'}\rightarrow b_{2}\rightarrow b_{3}\rightarrow...$, then $BT_{4}$ can be added into the path $P$ in this case; if $b_{2}\rightarrow3^{'}$, we consider two cases as follows.

$(a)$ When $2^{'}\rightarrow b_{3}$. We have $a_{1}\rightarrow b_{1}\rightarrow b_{2}\rightarrow3^{'}\rightarrow4^{'}\rightarrow1^{'}\rightarrow2^{'}\rightarrow b_{3}\rightarrow...$. Then $BT_{4}$ can be added into the path $P$ in this case.

$(b)$ When $b_{3}\rightarrow2^{'}$, then $4^{'}\rightarrow b_{3}$. As $2^{'}\rightarrow b_{1}$, we have $a_{1}\rightarrow 1^{'}\rightarrow 2^{'}\rightarrow b_{1}\rightarrow b_{2}\rightarrow 3^{'}\rightarrow 4^{'}\rightarrow b_{3}\rightarrow...$. Then $BT_{4}$ can be added into the path $P$ in this case.

Then $BT_{4k}$ contains $D(4k,p)$ for $2\leq p\leq 2k+1$ i.e. $BT_{4k}$ contains $D(4k,p)$ for $2\leq p\leq 4k$.

Case 2. $BT_{4(k-1)}$ is isomorphic to a digraph $D$ which has a Hamiltonian cycle $(1,2,3,...,4k-4,1)$ and for any vertex $i$, there are $(4m+i-1,i)\in A(D)$ and $(i,4m+i+1)\in A(D)$, where $1\leq m\leq k-2$, the vertices modulo $4k-4$ so that the vertex $4k-4+i$ is the vertex $i$.

$BT_{4k}$ is composed of $BT_{4}$ and $BT_{4(k-1)}$. If $BT_{4k}$ is isomorphic to a digraph $D$ which has a Hamiltonian cycle $(1,2,3,...,4k,1)$ and for any vertex $i$, there are $(4m+i-1,i)\in A(D)$ and $(i,4m+i+1)\in A(D)$, where $1\leq m\leq k-1$, the vertices modulo $4k$ so that the vertex $4k+i$ is the vertex $i$. The result has proved. If not, we shall show that $BT_{4k}$ contains $D(4k,p)$ for all $2\leq p\leq 4k$. By the same way as in the case 1, $BT_{4}$ can be added into the Hamiltonian cycle of $BT_{4(k-1)}$. Since $BT_{4k}$ is $k$-regular and $BT_{4(k-1)}$ is $(k-1)$-regular, so there exist four vertices $s,s+2,t,t+2$ in the Hamiltonian cycle of $BT_{4(k-1)}$ such that either $s\rightarrow 1^{'}\rightarrow (s+2)\rightarrow 3^{'}\rightarrow s$ or $(s+2)\rightarrow 1^{'}\rightarrow s\rightarrow 3^{'}\rightarrow (s+2)$ and either $t\rightarrow 2^{'}\rightarrow (t+2)\rightarrow 4^{'}\rightarrow t$ or $(t+2) \rightarrow 2^{'}\rightarrow t\rightarrow 4^{'}\rightarrow (t+2)$. We also have a cycle $s\rightarrow t\rightarrow (s+2)\rightarrow (t+2)\rightarrow s$ or $s\rightarrow (t+2)\rightarrow (s+2)\rightarrow t \rightarrow s$ since $BT_{4(k-1)}$ is isomorphic to a digraph $D$. Deleting $s,s+2,t,t+2$ from $BT_{4k}$, the resulting digraph is a decomposable $(k-1)$-regular bipartite tournament denoted by $BT_{4(k-1)}^{'}$ which is not isomorphic to $BT_{4(k-1)}$ and the four vertices form a decomposable $1$-regular bipartite tournament denoted by $BT_{4}^{'}$. Without loss of generality, we assume $BT_{4}^{'}$ has $s\rightarrow t\rightarrow (s+2)\rightarrow (t+2)\rightarrow s$ and $t>s+2$, we have a Hamiltonian cycle $(s-1)\rightarrow (s+4)\rightarrow...\rightarrow (t-1)\rightarrow (s+1)\rightarrow (t+1)\rightarrow (s+3)\rightarrow (t+3)\rightarrow....\rightarrow (s-1)$ in $BT_{4(k-1)}^{'}$. If $BT_{4(k-1)}^{'}$ is isomorphic to $BT_{4(k-1)}$, $BT_{4}$ must be added into two adjacent vertices of the Hamiltonian cycle of $BT_{4(k-1)}$. Without loss of generality, we assume $(t-1)\rightarrow 2^{'}$, then $2^{'}\rightarrow (t+1)\rightarrow 4^{'}\rightarrow (t-1)$ and $(s+3)\rightarrow 1^{'}\rightarrow (s+1)\rightarrow 3^{'}\rightarrow (s+3)$. We also have $(t-1)\rightarrow (s+1)\rightarrow (t+1)\rightarrow (s+3)\rightarrow (t-1)$. Now we select the four vertices $t-1,s+1,t+1,s+3$ instead of $s,s+2,t,t+2$. We get a decomposable $(k-1)$-regular bipartite tournament $BT_{4(k-1)}^{''}$ which is not isomorphic to $BT_{4(k-1)}$. So there must exist four vertices which form a decomposable $1$-regular bipartite tournament in $BT_{4(k-1)}$ such that, deleting them from $BT_{4k}$, the resulting digraph is a decomposable $(k-1)$-regular bipartite tournament denoted by $BT_{4(k-1)}^{*}$ which is not isomorphic to $BT_{4(k-1)}$. According to the assumptions, $BT_{4(k-1)}^{*}$ contains $D(4k-4,p)$ for $2\leq p\leq 4k-4$. By the similar way as in case 1, we have $BT_{4k}$ contains $D(4k,p)$ for all $2\leq p\leq 4k$.
The proof is complete.
\end{proof}

\section{Acknowledgements}

The research is supported by NSFC (No.11301371,  61502330), SRF for ROCS, SEM and Natural Sciences
Foundation of Shanxi Province (No. 2014021010-2), Fund Program for the Scientific Activities of Selected
Returned Overseas Professionals in Shanxi Province.

\end{document}